\documentclass{amsart}

\usepackage{amssymb}
\usepackage{amsthm}

 \newtheorem{theorem}{Theorem}[section]
 \newtheorem{corollary}[theorem]{Corollary}
 \newtheorem{lemma}[theorem]{Lemma}

 \newtheorem{rem}[theorem]{Remark}

\newcommand{\Z}{\mathbb{Z}}
\newcommand{\C}{\mathbb{C}}
\begin{document}

\title[Character degrees of symmetric groups]{\bf Symmetric groups are determined by  their character degrees}

\author{Hung P. Tong-Viet}
\email{Tong-Viet@ukzn.ac.za}
\address{School of Mathematical Sciences,
University of KwaZulu-Natal\\
Pietermaritzburg 3209, South Africa}

\date{\today}
\begin{abstract} Let $G$ be a finite group.  Let $X_1(G)$ be
the first column of the ordinary character table of $G.$ In this
paper, we will show that if $X_1(G)=X_1(S_n),$ then $G\cong S_n.$ As
a consequence, we show that $S_n$ is uniquely determined by the
structure of the complex group algebra $\C S_n.$

\end{abstract}

\thanks{Support from the University of KwaZulu-Natal is acknowledged}
\maketitle

\section{Introduction and Notations}
All groups considered are finite and all characters are complex
characters. Let $G$ be a group and let
$Irr(G)=\{\chi_1,\chi_2,\cdots,\chi_k\}$ be the set of all irreducible
characters of $G.$ Put $n_i=\chi_i(1).$ We say that
$(n_1,n_2,\cdots,n_k)$ is the \emph{degree pattern} of $G.$ Let
$cd(G)=\{ \chi(1)\:|\:\chi\in Irr(G)\}$ be the set of all
irreducible character degrees of $G.$ Following \cite{Ber1}, let
$X_1(G)$ be the first column of the ordinary character table of $G.$ By a
suitable re-ordering of the rows in the character table of $G,$ we
can see that $X_1(G)$ coincides with the degree pattern
$(n_1,n_2,\cdots,n_k)$ of $G.$ We also consider $X_1(G)$ as a
multiset consisting of character degrees of $G$ counting
multiplicities. Since $|G|=\sum_{\chi\in Irr(G)}\chi(1)^2,$ the
order of $G$ is known given  $X_1(G).$ There are examples showing
that non-isomorphic groups may have the same character table and so
the first column of their character tables coincide. Using the
classification of finite simple groups, it is easy to see that
non-abelian simple groups are uniquely determined by their character
table. It was shown by Nagao \cite{Nagao} that the symmetric groups
$S_n$ are also uniquely determined by their character tables. In
\cite{Hung}, we know that the alternating group $A_n$ of degree at
least $5,$ and the sporadic simple groups are uniquely determined by
the first column of their character tables. In this paper, we will
prove a similar result for the symmetric groups.

\begin{theorem}\label{main} Let $G$ be a finite
group. If $X_1(G)= X_1(S_n),$ then $G\cong S_n.$
\end{theorem}
This gives a positive answer to \cite[Question $126$]{Ber1}. Let
$\C$ be the complex number field and let $G$ be a group. Denote by
$\C G$ the group algebra of $G$ over $\C.$ Let $G_i,i=1,2,$ be
groups. By Molien's Theorem (\cite[Theorem $2.13$]{Ber1}) we know
that $\C G_1\cong \C G_2$ if and only if $X_1(G_1)=X_1(G_2).$
Therefore, knowing the first column of the character table of a
group $G$ is equivalent to knowing the structure of the group
algebra $\C G.$ It is known that $\C G$ allows us to recognize the
Frobenius groups or the $p$-nilpotent groups (\cite[Corollaries
$10.11,$ and $10.27$]{Ber1}). Now Theorem \ref{main} yields.
\begin{corollary}\label{main1} Let $G$ be a group.
If $\C G\cong \C S_n,$ then $G\cong S_n.$
\end{corollary}
We should mention that Brauer's Problem $1$ (see \cite{Brauer})
which asks the following: What are the possible degree patterns of
finite groups? Little is known about this problem. Now Corollary
\ref{main1} says that there is exactly one isomorphism type of the
group algebra with a degree pattern as that of the symmetric groups.

We now outline our argument for the proof of Theorem \ref{main}.
Assume that $X_1(G)=X_1(S_n).$ We first observe that
$|G:G'|=2,|G|=n!, k(G)=k(S_n)$ and $cd(G)=cd(S_n),$ where $k(G)$
denotes the number of conjugacy classes of $G.$ The result is
trivial when $n\leq 4.$ Hence we will assume that $n\geq 5.$ Next we
will show that $G'$ is perfect, that is $G'=G'',$ by applying
\cite[Lemma $12.3$]{Isaacs}. Choose $M\leq G'$ be a normal subgroup
of $G$ so that $G'/M$ is a chief factor of $G.$ As $|G:G'|=2$ and
$G'/M$ is non-abelian, we deduce that $G'/M\cong S^k,$ where $S$ is
a non-abelian simple group and $k$ is at most $2.$ We proceed to
show that $G'/M$ must be a simple group that is $k=1.$ This is done
by applying Theorem \ref{th2}. We now deduce that either $G/M$ is an
almost simple group with socle $G'/M$ or $G/M\cong G'/M\times \Z_2.$
We now apply Theorem \ref{th1} which asserts that if $H$ is an
almost simple group and $cd(H)\subseteq cd(S_n),n\geq 5,$ then the
socle of $H$ must be isomorphic to $A_n,$ to show that $G'/M\cong
A_n.$ Assume that $n\neq 6.$ By comparing the orders, $G\cong S_n$
or $G\cong A_n\times \Z_2.$ Finally, using the fact that $G$ and
$S_n$ have the same number of irreducible characters, we can
eliminate the latter case. Thus $G$ must be isomorphic to $S_n.$ In
the exceptional case, we have $|Out(A_6)|=4.$ In this case, $G$ is
one of the following groups: $A_6\times \Z_2, PGL_2(9)\cong
A_6.2_2,M_{10}\cong A_6.2_3$ or $S_6.$ Using \cite{atlas}, we
conclude that $G\cong S_6.$ We remark that this argument is based on
the Huppert's method given in \cite{Hupp}. This method is used to verify
the Huppert Conjecture which states that non-abelian simple groups are
determined by their sets of character degrees (see
\cite{Hupp,Wake}).

Here are some notation. If $cd(G)=\{s_0,s_1,\cdots,s_t\},$ where $1=s_0<s_1<\cdots<s_t,$ then we
define $d_i(G)=s_i$ for all $1\leq i\leq t.$  Then $d_i(G)$ is the $i^{th}$ smallest degree of the
non-trivial character degrees of $G.$ If $n$ is an integer then we denote by $\pi(n)$ the set of
all prime divisors of $n.$ If $G$ is a group, we will write $\pi(G)$ instead of $\pi(|G|)$ to
denote the set of all prime divisors of the order of $G.$ Let $p(G)=max(\pi(G))$ be the largest
prime divisor of the order of $G$ and let $\rho(G)=\cup_{\chi\in Irr(G)}\pi(\chi(1))$ be the set of
all primes which divide some irreducible character degrees of $G.$ Finally, if $N\unlhd G$ and
$\theta\in Irr(N),$ then the inertia group of $\theta$ in $G$ is denoted by $I_G(\theta).$ Other
notations are standard.
\section{Preliminaries}
Let $n$ be a positive integer. We call
$\lambda=(\lambda_1,\lambda_2,\dots, \lambda_r)$ a  partition of
$n,$ written $\lambda\vdash n,$ provided $\lambda_i, i=1,2,\dots, r$
are integers, with $\lambda_1\geq\lambda_2\geq \dots\geq
\lambda_r>0$ and $\sum_{i=1}^r \lambda_i=n.$ We collect the same
parts together and write
$\lambda=(\ell_1^{a_1},\ell_2^{a_2},\cdots,\ell_k^{a_k}),$ with
$\ell_i>\ell_{i+1}>0$ for $i=1,\cdots, k-1;a_i\neq 0;$ and
$\sum_{i=1}^k a_i\ell_i=n.$ It is well known that the irreducible
complex characters of the symmetric group $S_n$ are parametrized by
partitions of $n.$ Denote by $\chi^\lambda$  the irreducible
character of $S_n$ corresponding to partition $\lambda.$ The
irreducible characters of the alternating group $A_n$ are then
obtained by restricting $\chi^\lambda$ to $A_n.$ In fact,
$\chi^\lambda$ is still irreducible upon restriction to the
alternating group $A_n$ if and only if $\lambda$ is not
self-conjugate. Otherwise, $\chi^\lambda$ splits into two
irreducible characters of $A_n$ having the same degree. The
following result on the minimal character degrees of symmetric groups is
due to Rasala \cite{Ras}.
\begin{lemma}\emph{(\cite[Result $3$]{Ras}).}\label{lem1}
Let $\lambda$ be a partition of $n.$

$(a)$ If  $n\geq 15,$ then the first $6$ nontrivial minimal
character degrees of  $S_n$ are:

$(1)$ $d_1(S_n)=n-1$ and $\lambda\in\{(n-1,1),(2,1^{n-2})\};$

$(2)$ $d_2(S_n)=\frac{1}{2}n(n-3)$ and
$\lambda\in\{(n-2,2),(2^2,1^{n-4})\};$

$(3)$ $d_3(S_n)=d_2(S_n)+1=\frac{1}{2}(n-1)(n-2)$ and
$\lambda\in\{(n-2,1^2),(3,1^{n-3})\};$

$(4)$ $d_4(S_n)=\frac{1}{6}n(n-1)(n-5)$ and
$\lambda\in\{(n-3,3),(2^3,1^{n-6})\};$

$(5)$ $d_5(S_n)=\frac{1}{6}(n-1)(n-2)(n-3)$ and
$\lambda\in\{(n-3,1^3),(4,1^{n-4})\};$

$(6)$ $d_6(S_n)=\frac{1}{3}n(n-2)(n-4)$ and
$\lambda\in\{(n-3,2,1),(3,2,1^{n-5})\};$

$(b)$  If  $n\geq 22,$ then the next five smallest character degrees
are:

$(7)$ $d_7(S_n)=n(n-1)(n-2)(n-7)/24$ and
$\lambda\in\{(n-4,4),(2^4,1^{n-8})\};$

$(8)$ $d_8(S_n)=(n-1)(n-2)(n-3)(n-4)/24$ and
$\lambda\in\{(n-4,1^4),(5,1^{n-5})\};$

$(9)$ $d_9(S_n)=n(n-1)(n-4)(n-5)/12$ and
$\lambda\in\{(n-4,2^2),(3^2,1^{n-6})\};$

$(10)$ $d_{10}(S_n)=n(n-1)(n-3)(n-6)/8$ and
$\lambda\in\{(n-4,3,1),(3,2^2,1^{n-7})\};$

$(11)$ $d_{11}(S_n)=n(n-2)(n-3)(n-5)/8$ and
$\lambda\in\{(n-4,2,1^2),(4,2,1^{n-6})\};$
\end{lemma}

Assume $n\geq 5.$ Using \cite{atlas, GAP} and Lemma \ref{lem1}, we
can see that $d_1(S_n)=n-1$ and $d_2(S_n)=n(n-3)/2$ if $n\neq 8$
while $d_2(S_8)=14.$ Similarly, if $n\geq 6,$ then $d_1(A_n)=n-1$
while $d_1(A_5)=3$ (see \cite{Hung}). The following is well-known.
\begin{lemma}\emph{(Tschebyschef).}\label{lem2} If $m\geq 15,$ then
there is at least one prime $p$ with $m/2<p\leq m.$
\end{lemma}
\begin{proof} If $m\geq 17$ then the result follows from \cite[Proposition $5.1$]{Malle99}.
For $15\leq m\leq 16,$ the lemma is obvious.
\end{proof}

The following results on the classification of prime power character
degrees of symmetric groups will be used frequently.
\begin{lemma}\emph{(\cite[Theorem $5.1$]{bal}).}\label{lem3}
Suppose that $S_n$ possesses a non-trivial irreducible character
$\chi$ with $\chi(1)=p^d,$ where $p$ is a prime. Then one of the
following holds:

$(1)$ $n=p^d+1,$  and $\chi(1)=p^d;$

$(2)$ $n=4$ and $\chi(1)=2;$

$(3)$ $n=5$ and $\chi(1)=5;$

$(4)$ $n=6$ and $\chi(1)\in \{3^2,2^4\};$

$(5)$ $n=8$ and $\chi(1)=2^6;$

$(5)$ $n=9$ and $\chi(1)=3^3;$

\end{lemma}
We refer to \cite[$13.8,13.9$]{car85} for the classification of
unipotent characters and the notion of symbols.

\begin{lemma}\label{lem6} Let $S$ be a simple group of Lie type  in
characteristic $p$ defined over a field of size $q.$ Assume that
$S\neq L_2(q),{}^2F_4(2)'.$ Then there exist two irreducible
characters $\chi_i,i=1,2,$ of $S$ such that both $\chi_i$ extend to
$Aut(S)$ with $1<\chi_1(1)<\chi_2(1)$ and $\chi_2(1)=|S|_p.$ In
particular, if $G$ is an almost simple group with socle $S,$ where
$S\neq L_2(q),{}^2F_4(2)',$ then $|S|_p> d_1(G).$
\end{lemma}

\begin{proof} By the results of Lusztig \cite{Lusz}, any unipotent
character $\theta$  of $S$ has an extension $\tilde{\theta}$ to the
group $G_1$ of inner-diagonal automorphisms of $S$ such that
$\theta$ and $\tilde{\theta}$ have the same inertia group in
$Aut(S)$ (see \cite[Proposition $2.1$]{Malle08}). Moreover, the
unipotent characters of $G_1$ remain irreducible upon restriction to
$S,$ and these restrictions are all the unipotent characters of $S.$
By \cite[Theorem $2.4$]{Malle08}, all unipotent characters of $S$
extend to their inertia groups in $Aut(S).$ By results of Lusztig,
the inertia group of a unipotent character of $S$ is exactly
$Aut(S)$ except for several cases explicitly listed in \cite[Theorem
$2.5$]{Malle08}. Thus we can choose $\chi_2$ to be the Steinberg
character of $S$ and $\chi_1$ to be any unipotent character of $S$
such that $\chi_1$ does not appear in \cite[Theorem $2.5$]{Malle08}
and $1<\chi_1(1)<\chi_2(1)=|S|_p.$

Assume $S$ is of type $A_{n-1}$ with $n\geq 3.$ We have
$G_1=(A_{n-1})_{ad}(q)=PGL_n(q).$ By \cite[$13.8$]{car85}, the
unipotent characters of $G_1$ are parametrized by partitions of $n.$
Let $\alpha=(1,n-1).$  Then the degree of the unipotent character
$\chi^\alpha$ corresponding to $\alpha$ is given by
$\chi^\alpha(1)=(q^n-q)/(q-1).$  Since $St_S(1)=|S|_p=q^{n(n-1)/2},$
and $n\geq 3,$ we have $St_S(1)>\chi^\alpha(1)>1.$

Assume $S$ is of type ${}^2A_{n-1},$ where $n\geq 3.$ Then
$G_1=({}^2A_{n-1})_{ad}(q^2)=PU_n(q).$ By \cite[13.8]{car85}, the
unipotent characters of $G_1$ are again parametrized by partitions
of $n.$ Let $\alpha=(1,n-1).$ Then the degree of the unipotent
character $\chi^\alpha$ corresponding to $\alpha$ is given by
$\chi^\alpha(1)=(q^n+(-1)^nq)/(q+1).$  Since
$St_S(1)=|S|_p=q^{n(n-1)/2},$ and $n\geq 3,$ we have
$St_S(1)>\chi^\alpha(1)>1.$

Assume next that $S$ is of type $B_n,$ or $C_n$ where $n\geq 2$ and
$S\neq S_4(2).$ Then $G_1=(B_n)_{ad}(q)=SO_{2n+1}(q)$ or
$G_1=(C_n)_{ad}(q)=PCSp_{2n}(q).$ By \cite[13.8]{car85}, $G_1$ has a
unipotent characters $\chi^\alpha$ labeled by the symbol
$$\alpha=\begin{pmatrix}
      0 & 1 & n  \\
       \: & - &\:
\end{pmatrix}$$

with $\chi^\alpha(1)=(q^n-1)(q^n-q)/(2(q+1)).$ Since $|S|_p=q^{n^2}$
and $(n,q)\neq (2,2),$ we see that $|S|_p>\chi^\alpha(1)>1.$

Assume $S$ is of type $D_n(q),n\geq 4.$ Then
$G_1=(D_n)_{ad}(q)=P(CO_{2n}(q)^0).$ By \cite[13.8]{car85}, $G_1$
has a unipotent character $\chi^\alpha$ labeled by the symbol
$$\alpha=\begin{pmatrix}
      n-1  \\
      1
\end{pmatrix}$$ with $\chi^\alpha(1)=(q^n-1)(q^{n-1}+q)/(q^2-1).$
Since $|S|_p=q^{n(n-1)},$ we see that $|S|_p>\chi^\alpha(1)>1.$

Assume $S$ is of type ${}^2D_n(q^2),n\geq 4.$ Then
$G_1=({}^2D_n)_{ad}(q^2)=P(CO_{2n}^-(q)^0)$ and $G_1$ has a
unipotent character $\chi^\alpha$  labeled by the symbol
$$\alpha=\begin{pmatrix}
      1 &&n-1  \\
      & - &
\end{pmatrix}$$ with $\chi^\alpha(1)=(q^n+1)(q^{n-1}-q)/(q^2-1).$
Since $|S|_p=q^{n(n-1)},$  $|S|_p>\chi^\alpha(1)>1.$

For the simple groups of exceptional type, we will use the explicit
list of unipotent characters in \cite[$13.9$]{car85}.

Assume $S$ is of type $G_2(q).$  Then $S$ has a unipotent character
labeled by $\theta_{2,1}$ with degree $q\Phi_2^2\Phi_3/6.$ As
$G_2(2)\cong U_3(3).2$ is not simple, we can assume that $q\geq 3.$
Since $|S|_p=q^6,$ we have $q\Phi_2^2\Phi_3/6<q^6$ so that
$1<\theta_{2,1}(1)<|S|_p.$

Assume $S$ is of type ${}^3D_4(q^3).$
Then $S$ has a unipotent character labeled by $\theta_{1,3'}$ with
degree $q\Phi_{12}.$  Since $|S|_p=q^{12},$ we have
$1<\theta_{1,3'}(1)<|S|_p.$

Assume $S$ is of type $F_4(q).$ Then $S$ has a unipotent character
labeled by $\theta_{9,2}$ with degree
$q^2\Phi_3^2\Phi_6^2\Phi_{12}.$ Since $|S|_p=q^{24},$ we have
$1<\theta_{9,2}(1)<|S|_p.$

Assume $S$ is of type $E_6(q).$  Then $S$ has a unipotent character
labeled by $\theta_{6,1}$ with degree $q\Phi_8\Phi_9.$ Since
$|S|_p=q^{36},$ we have $1<\theta_{6,1}(1)<|S|_p.$

Assume $S$ is of type ${}^2E_6(q^2).$ Then $S$ has a unipotent
character labeled by $\theta_{2,4'}$ with degree $q\Phi_8\Phi_{18}.$
Since $|S|_p=q^{36},$ we have $1<\theta_{2,4'}(1)<|S|_p.$

Assume $S$ is of type $E_7(q).$  Then $S$ has a unipotent character
labeled by $\theta_{7,1}$ with degree $q\Phi_7\Phi_{12}\Phi_{14}.$
Since $|S|_p=q^{63},$ we have $1<\theta_{7,1}(1)<|S|_p.$

Assume $S$ is of type $E_8(q).$  Then $S$ has a unipotent character
labeled by $\theta_{8,1}$ with degree
$q\Phi_4^2\Phi_8\Phi_{12}\Phi_{20}\Phi_{24}.$ Since $|S|_p=q^{120},$
we have $1<\theta_{8,1}(1)<|S|_p.$

Assume $S$ is of type ${}^2B_2(q^2),$ where $q^2=2^{2m+1}$ and
$m\geq 1.$  Then $S$ has a unipotent character labeled by
${}^2B_2[a]$ with degree $q\Phi_1\Phi_2/\sqrt{2}.$ Since
$|S|_p=q^{4},$ we have $1<{}^2B_2[a](1)<|S|_p.$

Assume $S$ is of type ${}^2G_2(q^2),$ where $q^2=3^{2m+1}$ and
$m\geq 1.$  Then $S$ has a unipotent character $\theta$ with degree
$q\Phi_1\Phi_2\Phi_4/\sqrt{3}.$ Since $|S|_p=q^{6},$ we have
$1<\theta(1)<|S|_p.$

Assume $S$ is of type ${}^2F_4(q^2),$ where $q^2=2^{2m+1}$ and
$m\geq 1.$  Then $S$ has a unipotent character labeled by
${}^2B_2[a]$ with degree $q\Phi_1\Phi_2\Phi_4^2\Phi_6/\sqrt{2}.$
Since $|S|_p=q^{24},$ we have $1<{}^2B_2[a](1)<|S|_p.$ This finishes
the proof of the first assertion.

Now assume $G$ is an almost simple group with socle $S,$ where
$S\neq L_2(q),{}^2F_4(2)'.$ Let $\chi_i\in Irr(S),i=1,2,$ be
irreducible characters of $S$ obtained above. As both $\chi_i$
extend to $Aut(S)$ and $S\unlhd G\leq Aut(S),$ we deduce that each
$\chi_i$ extends to $\psi_i\in Irr(G)$ with $\psi_i(1)=\chi_i(1)$
for $i=1,2.$ Thus $\psi_2(1)=|S|_p>\psi_1(1)=\chi_1(1)>1$ so that
$|S|_p>d_1(G)$ as required. The proof is now complete.
\end{proof}

\begin{lemma}\label{lem8} Let $G$ be an almost simple group with
socle $S=L_2(q),$ where $q=p^f\geq 7.$ If $p\neq 3,$ then $G$
contains an irreducible character of degree $q+\delta$ where
$q\equiv \delta\mbox{\:\emph{(mod $3$)}}.$ If $p=3,$ then $G$
contains an irreducible character of degree $(q+\epsilon)/2$ or
$q+\epsilon,$ where $q\equiv \epsilon\mbox{\:\emph{(mod $4$)}}.$
\end{lemma}
\begin{proof} Assume first that $p\neq 3.$ It follows from the proof
of \cite[Proposition $3.7$]{MM} that the irreducible character  of
$S$ corresponding to a semisimple element of order $3$ in the dual
group $SL_2(q),$  of degree $q+\delta,$ where $q\equiv
\delta\mbox{\:{(mod $3$)}},$ is extendible to $Aut(S),$ and hence
$G$ contains an irreducible character of degree $q+\delta$ as
required. Note that this result fails for $L_2(3^f).$ Now we assume
that $q=3^f.$ Observe that $L_2(q)$ always contains irreducible
characters $\chi_a,\chi_b$ of degree $q-1$ and $q+1,$ respectively,
which are extendible to $PGL_2(q).$ Thus if $L_2(q)\unlhd G\leq
PGL_2(q)$ then $G$ possesses characters of degree $q\pm 1.$ Now
assume that $G\not\leq PGL_2(q).$ Recall that the only outer
automorphisms of $L_2(q)$ are the diagonal automorphisms and the
field automorphisms. It is well-known that $S$ contains two
irreducible characters $\chi^{\pm}$ of degree $(q+\epsilon)/2,$
where $q\equiv \epsilon \mbox{\:{(mod $4$)}}.$ Now the diagonal
automorphisms of $S$ fuse these two irreducible characters while the
field automorphisms fix those two. Let $\theta\in
\{\chi^+,\chi^-\}.$ Then $\theta\in Irr(S)$ and
$I_{Aut(S)}(\theta)=P\Gamma L_2(q).$ Thus if $S\unlhd G\leq P\Gamma
L_2(q),$ then $\theta$ is $G$-invariant and so $\theta$ extends to
$G$ as $G/S$ is cyclic. Hence $G$ has an irreducible character of
degree $(q+\epsilon)/2.$ Finally, assume $PGL_2(q)\unlhd G\leq
Aut(L_2(q)).$ Then the irreducible character $\mu$ of $PGL_2(q)$
lying over $\theta$ is of degree $q+\epsilon.$ We see that $\mu$ is
$G$-invariant and hence it extends to $G,$ as $G/PGL_2(q)$ is
cyclic. Therefore $G$ contains an irreducible character of degree
$q+\epsilon.$ The proof is now complete.
\end{proof}

\begin{table}
 \begin{center}
  \caption{Sporadic simple groups and their automorphism groups}\label{Tab1}
  \begin{tabular}{c|c|c|c|c}
   \hline
   $S$  & $p(S)$ & $d_1(S)$&$d_2(S)$&$d_3(S)$\\ \hline\hline
   $M_{11}$&$11$&$10$&$11$&$16$\\\hline
   $M_{12}$&$11$&$11$&$16$&$45$\\\hline
   $M_{12}.2$&$11$&$22$&$32$&$45$\\\hline
   $J_1$&$19$&$56$&$76$&$77$\\\hline
   $M_{22}$&$11$&$21$&$45$&$55$\\\hline
   $M_{22}.2$&$11$&$21$&$45$&$55$\\\hline
   $J_2$&$7$&$14$&$21$&$36$\\\hline
   $J_2.2$&$7$&$28$&$36$&$42$\\\hline
   $M_{23}$&$23$&$22$&$45$&$230$\\\hline
   $HS$&$11$&$22$&$77$&$154$\\\hline
   $HS.2$&$11$&$22$&$77$&$154$\\\hline
   $J_3$&$19$&$85$&$323$&$324$\\\hline
   $J_3.2$&$19$&$170$&$324$&$646$\\\hline
   $M_{24}$&$23$&$23$&$45$&$231$\\\hline
   $McL$&$11$&$22$&$231$&$252$\\\hline
   $McL.2$&$11$&$22$&$231$&$252$\\\hline
   $He$&$17$&$51$&$153$&$680$\\\hline
   $He.2$&$17$&$102$&$306$&$680$\\\hline
   $Ru$&$29$&$378$&$406$&$783$\\\hline
   $Suz$&$13$&$143$&$364$&$780$\\\hline
    $Suz.2$&$13$&$143$&$364$&$780$\\\hline
   $O'N$&$31$&$10944$&$13376$&$25916$\\\hline
    $O'N.2$&$31$&$10944$&$26752$&$37696$\\\hline
   $Co_3$&$23$&$23$&$253$&$275$\\\hline
   $Co_2$&$23$&$23$&$253$&$275$\\\hline
   $Fi_{22}$&$13$&$78$&$429$&$1001$\\\hline
    $Fi_{22}.2$&$13$&$78$&$429$&$1001$\\\hline
   $HN$&$19$&$133$&$760$&$3344$\\\hline
   $HN.2$&$19$&$266$&$760$&$3344$\\\hline
   $Ly$&$67$&$2480$&$45694$&$48174$\\\hline
   $Th$&$31$&$248$&$4123$&$27000$\\\hline
   $Fi_{23}$&$23$&$782$&$3588$&$5083$\\\hline
   $Co_1$&$23$&$276$&$299$&$1771$\\\hline
   $J_4$&$43$&$1333$&$299367$&$887778$\\\hline
   $Fi_{24}'$&$29$&$8671$&$57477$&$249458$\\\hline
   $Fi_{24}'.2$&$29$&$8671$&$57477$&$249458$\\\hline
   $B$&$47$&$4371$&$96255$&$1139374$\\\hline
   $M$&$71$&$196883$&$21296876$&$842609326$\\\hline
   ${}^2F_4(2)'$&$13$&$26$&$27$&$78$\\\hline
   ${}^2F_4(2)'.2$&$13$&$27$&$52$&$78$\\\hline
  \end{tabular}

 \end{center}
\end{table}

\begin{corollary}\label{cor3} If $G$ is an almost simple group then $\rho(G)=\pi(G).$
\end{corollary}

\begin{proof} Observe first that for any $\chi\in Irr(G),$ we have $\chi(1)$ divides $|G|$ by
\cite[Theorem $3.11$]{Isaacs}. Hence $\rho(G)\subseteq \pi(G).$ As
$G$ is almost simple, it has no normal abelian Sylow $p$-subgroup,
so that by the Ito-Michler Theorem \cite{Mich}, every prime divisor
of $G$ must divide $\chi(1)$ for some $\chi\in Irr(G),$ and thus
$\pi(G)\subseteq \rho(G).$ Hence $\rho(G)=\pi(G)$ as required.
\end{proof}

\begin{lemma}\label{lem7} Let $G$ and $H$ be groups. Suppose that $cd(G)\subseteq cd(H).$ Then

$(i)$ $d_i(G)\geq d_i(H),$ for all $i\geq 1;$

$(ii)$ If $G$ is almost simple then $\pi(G)\subseteq \pi(H).$
\end{lemma}
\begin{proof} $(i)$ is obvious. $(ii)$ follows from Corollary
\ref{cor3} as $\rho(G)\subseteq \rho(H)\subseteq \pi(H).$
\end{proof}

\section{Proofs of the main results}

\begin{theorem}\label{th1} Let $G$ be an almost simple group with socle $S$ and let $n\geq 5$ be an integer.
If $cd(G)\subseteq cd(S_n)$ then $S\cong A_n.$
\end{theorem}
\begin{proof}  Using the classification of finite simple groups, $S$
is an alternating group of degree at least $5,$ a finite simple
group of Lie type or one of the $26$ sporadic groups. We will treat
the Tits group as a sporadic group rather than a group of Lie type.

{\bf Step $1.$} Eliminate simple groups of Lie type. By way of
contradiction, we assume that $S$ is a simple group of Lie type in
characteristic $p$ and $cd(G)\subseteq cd(S_n)$ with $n\geq 5.$ By
the isomorphisms $L_2(4)\cong L_2(5)\cong A_5, L_2(9)\cong A_6$ and
$L_4(2)\cong A_8,$ we can assume that $S$ is not one of the groups
listed above nor the Tits group. It is well known that the Steinberg
character of $S$ of degree $|S|_p$ extends to $\chi\in Irr(G)$ and
hence $\chi(1)=|S|_p$ is a non-trivial power of $p.$ Assume first
that $|S|_p$ is not the minimal character degree of $S_n,$ that is
$|S|_p>n-1.$ It follows from Lemma \ref{lem3} that $n\in
\{5,6,8,9\}$ and $|S|_p=5,|S|_p\in \{3^2,2^3\},|S|_p=2^6,|S|_p=3^3,$
respectively. It is routine to check that these cases cannot happen.
Thus $|S|_p=n-1=d_1(S_n).$ Now Lemma \ref{lem6} will provide a
contradiction unless $S=L_2(q),$ where $q=p^f\geq 4.$ Assume that
$S=L_2(q)$ and thus $q\geq 7.$ As $|S|_p=q=d_1(S_n),$ by Lemma
\ref{lem8}, $G$ must contain an irreducible character of degree
$q+1$ and hence $q+1\in cd(S_n).$ Since $S\neq L_4(2)\cong A_8,$ we
have $d_2(S_n)=n(n-3)/2.$ We have $n-1=q$ and hence as $q\geq 7,$ we
obtain $d_2(S_n)=n(n-3)/2=(q+1)(q-2)/2>q+1>q=d_1(S_n),$ which
contradicts Lemma \ref{lem7}$(i).$ This finishes the proof of Step
$1.$

{\bf Step $2.$}  Eliminate sporadic simple groups and the Tits
group. By way of contradiction, we assume that $S$ is a simple
sporadic group or the Tits group and that $cd(G)\subseteq cd(S_n)$
with $n\geq 5.$ The character degrees of $S_n,$ where $5\leq n\leq
31$ can be found in \cite{GAP}. Moreover the character degrees of
sporadic simple groups and the Tits group together with their
automorphism groups are also available in \cite{GAP}. It is routine
to check that $cd(G)\nsubseteq cd(S_n)$ for any $5\leq n\leq 31$ and
any almost simple group $G$ with socle $S,$ where $S$ is a sporadic
simple group or the Tits group. Thus we can assume that $n\geq 32.$
It follows from Lemma \ref{lem1} that $d_2(S_n)=n(n-3)/2\geq
d_2(S_{32})=464.$ By Lemma \ref{lem7}$(i)$ and Table \ref{Tab1}, we
only need to consider the following cases: $S\in \{O'N, HN, Ly, Th,
Fi_{23},J_4,Fi_{24}',B,M\}.$

$(1)$ $S=O'N.$ In this case, we have $|Out(S)|=2$ so that either
$G=S$ or $G=S.2.$ Assume first that $G=S=O'N.$ Then $d_9(O'N)=58653$
and since $n\geq 32,$ by Lemma \ref{lem1}, $d_9(S_n)\geq 62496>
d_9(O'N),$ which contradicts Lemma \ref{lem7}$(i).$ Now assume
$G=O'N.2.$ Then $d_7(G)=58653<62496\leq d_9(S_n)$ so that $d_7(G)\in
\{d_7(S_n),d_8(S_n)\}.$ However we can check that these equations
cannot hold for any $n\geq 32.$ Thus $cd(G)\nsubseteq cd(S_n).$

$(2)$ $S=HN.$ Then $|Out(S)|=2$ so that $G=S$ or $G=S.2.$ From
\cite{atlas}, we have $d_7(S)=16929$ and $d_7(S.2)=17556.$ Observe
that $d_7(G)<62496\leq d_7(S_n)$ so that $cd(G)\nsubseteq cd(S_n)$
by Lemma \ref{lem7}$(i).$

$(3)$ $S=Ly.$ Since $|Out(S)|=1,$ we have $G=S$ so that $p(S)=67\in
\pi(S_n)$ by Lemma \ref{lem7}$(ii),$ and hence $n\geq 67.$ As
$d_5(Ly)=381766<718575\leq d_7(S_n),$ we deduce that $d_5(Ly)\in
\{d_5(S_n),d_6(S_n)\}.$ However we can check that these equations
cannot hold for any $n\geq 67.$ Thus $cd(G)\nsubseteq cd(S_n).$

$(4)$ $S=Th.$ As $Out(S)$ is trivial, we have $G=S.$ Since
$d_1(G)=248<464\leq d_2(S_n),$ it follows from Lemma \ref{lem7}$(i)$
that $d_1(G)=d_1(S_n)=n-1$ and hence $n=249.$ But then
$d_2(S_n)=n(n-3)/2\geq 30627>d_2(Th).$ Thus $cd(G)\nsubseteq
cd(S_n).$

$(5)$ $S=Fi_{23}.$ As $Out(S)$ is trivial, we have $G=S.$ Since
$d_2(G)=3588<4464\leq d_4(S_n),$ it follows from Lemma
\ref{lem7}$(i)$ that $d_2(G)\in \{d_2(S_n),d_3(S_n)\}.$ However we
can check that these equations cannot hold for any $n\geq 32.$ Thus
$cd(G)\nsubseteq cd(S_n).$

$(6)$ $S=J_4.$ Since $|Out(S)|=1,$ we have $G=S$ so that $p(S)=43\in
\pi(S_n)$ and hence $n\geq 43.$ As $d_1(J_4)=1333<11438\leq
d_4(S_n),$ we deduce that $d_1(J_4)\in \{d_i(S_n)\;|\;i=1,2,3\}.$
Solving these equations, we obtain $n=1334.$ But then
$d_2(S_n)=887777>299367=d_2(J_4).$ Thus $cd(G)\nsubseteq cd(S_n).$

$(7)$ $S=Fi_{24}'.$ We have $G=S$ or $G=S.2.$ In both cases, we have
$d_1(G)=8671$ and $d_2(G)=57477.$ Observe that $d_1(G)=8671<8960\leq
d_6(S_n)$ so that $d_1(G)\in \{d_i(S_n)\;|\;i=1,\cdots,5\}.$ Solving
these equations, we obtain $n=8672.$ But then $d_2(S_n)>d_2(G).$
Thus $cd(G)\nsubseteq cd(S_n).$

$(8)$ $S=B.$ Since $|Out(S)|=1,$ we have $G=S$ so that $p(S)=47\in
\pi(S_n)$ and hence $n\geq 47.$ As $d_1(B)=4371<15134\leq d_4(S_n),$
we deduce that $d_1(B)\in \{d_1(S_n),d_2(S_n),d_3(S_n)\}.$ Solving
these equations, we have $n=95$ or $n=4372.$ If the latter case
holds then $d_2(S_n)>d_2(B),$ a contradiction. Thus $n=95.$ But then
$d_3(S_{95})=4371<d_2(B)<d_4(S_{95})=133950.$ Thus $cd(G)\nsubseteq
cd(S_n).$

$(9)$ $S=M.$ Since $|Out(S)|=1,$ we have $G=S$ so that $p(S)=71\in
\pi(S_n)$ and hence $n\geq 71.$ As $d_1(M)=196883<914480\leq
d_7(S_n),$ we deduce that $d_1(M)\in \{d_i(S_n)\;|\;i=1,\cdots,6\}.$
Solving these equations, we obtain $n=196884.$ But then
$d_2(S_n)>21296876=d_2(M).$ Thus $cd(M)\nsubseteq cd(S_n).$

{\bf Step $3.$}  If $S\cong A_m,$ with $m\geq 5,$ then $m=n.$ Let
$\lambda=(m-1,1)$ be a partition of $m.$ Since $m\geq 5,$ $\lambda$
is not self-conjugate so that the irreducible character
$\chi^\lambda$ of $S_m$ corresponding to $\lambda$ is still
irreducible upon restriction to $A_m.$ Note that $Aut(A_m)=S_m$
whenever $m\neq 6$ while $|Out(A_6)|=4.$ Assume first that $m\neq
6.$ Then $G\in\{A_m,S_m\}$ and $G$ contains an irreducible character
of degree $m-1.$ Since $cd(G)\subseteq cd(S_n),$ we have $m-1\geq
d_1(S_n)=n-1$ so that $m\geq n.$ If $m=n$ then we are done. Hence we
assume that $m>n\geq 5.$ It follows that $m\geq 6$ and hence
$d_1(A_m)=d_1(S_m)=m-1$ and thus $d_1(G)=m-1>n-1=d_1(S_n).$ If
$m\leq 17$ then $5\leq n<m\leq 17.$ Using \cite{GAP}, we can check
that $m=n.$ So we can assume that $m\geq 18.$ It follows that $17\in
\pi(G)$ and so by Lemma \ref{lem7}$(ii)$ we have $17\in \pi(S_n),$
which implies that $n\geq 17.$ Thus $17\leq n<m$. It follows from
Lemma \ref{lem1} that $d_2(S_n)=n(n-3)/2.$ Since $cd(G)\subseteq
cd(S_n)$ and $d_1(G)>d_1(S_n),$ it follows that $d_1(G)\geq
d_2(S_n).$ Then $m-1\geq n(n-3)/2.$ Since $n\geq 17,$ we have
$m-2n\geq n(n-3)/2+1-2n=n(n-7)/2+1>0$ so that $m>2n.$ Therefore
$n<m/2<m.$ By Lemma \ref{lem2}, there exists a prime $p$ such that
$m/2\leq p<m.$ It follows that $p\in \pi(G)$ but $p\not\in \pi(S_n)$
since $p>n,$ which contradicts Lemma \ref{lem7}$(ii).$  Thus $S\cong
A_n$ whenever $m\geq 5,m\neq 6.$ Now assume that $m=6,$ and
$A_6\unlhd G\leq Aut(A_6).$ It follows that $G\in \{A_6,A_6.2_1\cong
S_6, A_6.2_2\cong PGL_2(9),A_6.2_3\cong M_{10},{A_6}.2^2\}.$ We need
to show that $n=6.$ If $G\in \{A_6,S_6\},$ then $G$ contains a
character of degree $5$ so that $5\geq n-1$ and hence $n\leq 6.$ As
$10\in cd(G)$ but $10\not\in cd(S_5),$ we conclude that $n=6.$ If
$G\cong PGL_2(9)$ then $\{8,9\}\subseteq cd(S_n).$ But this cannot
happen by Lemma \ref{lem3}. Assume that one of the last two cases
holds. Then $\{9,16\}\subseteq cd(S_n)$ so that by Lemma \ref{lem3},
$n=6.$ The proof is now complete.
\end{proof}
\begin{rem} Let $\lambda=(k+1,1^k)$ when $n=2k+1$ and $\lambda=(k,2,1^{k-2})$
when $n=2k.$ Then $\lambda$ is a self-conjugate partition of $n.$ We
conjecture that $\chi^{\lambda}(1)/2\in cd(A_n)-cd(S_n)$ and
$\chi^{\lambda}(1)\in cd(S_n)-cd(A_n).$ If this conjecture is true
then, in the situation of Theorem \ref{th1}, we deduce that $G\cong
S_n$ or $G\cong M_{10}$ and $n=6.$ This result will be useful in
studying the Huppert Conjecture for alternating groups.

\end{rem}


\begin{theorem}\label{th2} Let $G$ be a group. Assume that $|G:G'|=2$ and that $G'\cong S^2$ is a unique minimal normal subgroup of
$G,$ where $S$ is a non-abelian simple group. Then $cd(G)\nsubseteq
cd(S_n)$ for any $n\geq 5.$
\end{theorem}

\begin{proof} By way of contradiction, assume that $cd(G)\subseteq cd(S_n).$
Let $\alpha\in Irr(S)$ with $\alpha(1)>1$ and put
$\theta=\alpha\times 1\in Irr(G').$ Observe that  $\theta$ is not
$G$-invariant so that $I_G(\theta)=G'$ hence $\theta^G\in Irr(G)$
and so $\theta^G(1)=2\alpha(1)\in cd(S_n).$ On the other hand, if
$\varphi=\alpha\times \alpha\in Irr(G')$ then $\varphi$ is
$G$-invariant and since $G/G'$ is cyclic, we deduce that $\varphi$
extends to $\psi\in Irr(G),$ so that $\psi(1)=\alpha(1)^2\in
cd(S_n).$ We conclude that if $a\in cd(S)-\{1\}$ then $2a,a^2\in
cd(S_n).$ Let $r\in \pi(S).$ By Corollary \ref{cor3}, $r|a$ for some
$a\in cd(S)-\{1\}.$ Since $a^2\in cd(S_n),$ by \cite[Theorem
$3.11$]{Isaacs} we have $a^2|n!.$ Thus $r^2|n!$ so that $n\geq 2r.$
Using the classification of finite simple groups, we consider the
following cases.

Case $S=A_m,$ with $m\geq 5.$ As $m\geq 5,$ it follows from the
first paragraph that $n\geq 10.$ Assume first that $m\in
\{5,6,8,9,10\}.$ Observe that $m-1\in cd(S)$ and hence
$2(m-1),(m-1)^2\in cd(S_n).$ For these values of $m,$ we see that
$(m-1)^2$ is a prime power and so by Lemma \ref{lem3}, as $n\geq
10,$ we have $d_1(S_n)=n-1=(m-1)^2.$ As $m\geq 5,$ we obtain
$d_1(S_n)=(m-1)^2>2(m-1)>1,$ which is a contradiction since
$2(m-1)\in cd(S_n).$ Now assume that $m=7.$ As above, we have $n\geq
14.$ As $6\in cd(S),$ we obtain $6.2=12\in cd(S_n)$ and so $12\geq
d_1(S_n)=n-1$ which implies $n\leq 13,$ a contradiction. Thus we can
assume that $m\geq 11$ and hence $n\geq 22.$ We have $m-1\in cd(S)$
so that $2(m-1)$ and $(m-1)^2$ are both in $cd(S_n).$ Similarly, by
Lemma \ref{lem1}, we have $m(m-3)/2,(m-1)(m-2)/2\in cd(S)$ and so
$m(m-3),(m-1)(m-2)\in cd(S_n).$ We will show that $m<n.$ By way of
contradiction, assume that $m\geq n.$ As $n\geq 22,$ by Lemma
\ref{lem2}, there exists a prime $r$ such that $n/2<r\leq n.$ It
follows that the largest power of $r$ dividing $n!$ is $r.$  Since
$r\leq n\leq m,$ we deduce that $r\in \pi(A_m)$ and so $r^2|n!,$
which is a contradiction. Thus $m<n.$ Observer that
$1<2(m-1)<m(m-3)<(m-1)(m-2)<(m-1)^2,$ since $m\geq 11.$ Thus
$(m-1)^2\geq d_4(S_n)=n(n-1)(n-5)/6$  by Lemma \ref{lem1}. Combining
with the fact that $m< n,$ we obtain $n(n-1)(n-5)/6\leq (n-1)^2$ so
that $n(n-5)\leq 6(n-1)$ or equivalently $n(n-11)+6\leq 0,$ which is
impossible as $n\geq 22.$

Case $S$ is a finite simple group of Lie type in characteristic $p,$
with $S\neq {}^2F_4(2)'.$ Since $L_2(4)\cong L_2(5)\cong A_5,$ we
can assume that $S\not\cong L_2(4).$ Let $St$ be the Steinberg
character of $S.$ We can check that $St(1)=|S|_p\geq 5.$ Since
$St(1)\in cd(S),$ we obtain $2St(1)\in cd(S_n)$ and $St(1)^2\in
cd(S_n).$ By Lemma \ref{lem3}, assume first that $n-1=St(1)^2.$ Then
$d_1(S_n)=St(1)^2>2St(1),$ which is a contradiction. Now assume that
$St(1)^2\neq n-1.$ By Lemma \ref{lem3}, $n\in \{5,6,8,9\}.$ Since
$St(1)^2$ is an even prime power, one of the following cases holds:
$n=6,$ $St(1)^2\in \{3^2,2^4\}$ or $n=8,St(1)^2=2^6.$ Assume first
that $n=6.$ Then $St(1)\in \{3,2^2\}$ which implies that $St(1)\leq
4,$ a contradiction as $S\neq L_2(4).$ Finally, assume that $n=8.$
Then $St(1)=2^3=8.$ However $2St(1)=2^4\not\in cd(S_8),$ a
contradiction.

Case $S$ is a sporadic simple group or the Tits group. Recall that
$p(S)$ is the largest prime divisor of $S.$ We have $n\geq 2p(S).$

\noindent $(1)$ $S\in
\{M_{11},M_{12},J_1,M_{22},J_2,M_{23},HS,J_3,M_{24},He,Ru,Co_3,Co_2,Co_1,{}^2F_4(2)'\}.$
These cases can be eliminated as follows: Since $n\geq 2p(S)\geq
14,$ we have $d_2(S_n)=n(n-3)/2\geq p(S)(2p(S)-3).$ Next observe
that $2d_i(S)\in cd(G)\subseteq cd(S_n),$ for $i=1,2$ and that
$1<2d_1(S)<2d_2(S).$ In each case, we have $p(S)(2p(S)-3)>2d_2(S)$
and hence $d_2(S_n)>2d_2(S)>2d_1(S)>1,$ which contradicts Lemma
\ref{lem7}$(i).$

\noindent $(2)$ $S\in \{McL, Suz, Fi_{22},HN,Ly,Th,J_4,B\}.$ Since
$n\geq 2p(S)\geq 14,$ we have $d_2(S_n)=n(n-3)/2\geq p(S)(2p(S)-3).$
We have $d_2(S_n)\geq p(S)(2p(S)-3)>2d_1(S)$ and so
$2d_1(S)=d_1(S_n)=n-1$ so that $n=2d_1(S)+1.$ But then
$d_2(S_n)=n(n-3)/2=(d_1(S)-1)(2d_1(S)+1)>2d_2(S)>2d_1(S)>1,$ which
contradicts Lemma \ref{lem7}$(i)$ as $2d_i(S)\in cd(G)\subseteq
cd(S_n),$ where $i=1,2.$

\noindent $(3)$ $S=O'N.$ Then $p(S)=31.$ We have $n\geq 2p(S)=62$
and so by Lemma \ref{lem1}, $d_7(S_n)\geq 520025.$ As
$d_8(S)=58311,$ we have $2d_8(S)=116622\in cd(S_n).$ Note that
$2d_i(S)\in cd(S_n)$ for $i=1,2,\cdots, 8$ and so $2d_i(S)\geq
d_i(S_n)$ for all $1\leq i\leq 8.$  As $d_7(S_n)\geq
520025>116622=2d_8(S),$ we get a contradiction.

\noindent $(4)$ If $S=Fi_{23}$ then $p(S)=23.$ We have $n\geq
2p(S)=46$ and so by Lemma \ref{lem1}, $d_4(S_n)\geq 14145.$ As
$d_2(S)=3588,$ we obtain $2d_2(S)=7176\in cd(S_n).$ Since
$d_4(S_n)>7176>2d_1(S),$ we must have $7176\in
\{d_2(S_n),d_3(S_n)\}.$ However, we can check that these cases
cannot happen.

\noindent $(5)$ $S=Fi_{24}'.$ Then $p(S)=29$ and $n\geq 2p(S)=58$ so
that by Lemma \ref{lem1}, $d_4(S_n)\geq 29203.$ As
$\{8671,57477\}\subseteq cd(S),$ we obtain
$\{17342,114954\}\subseteq cd(S_n).$ Since $d_4(S_n)>17342,$ we have
$17342\in \{d_1(S_n),d_2(S_n),d_3(S_n)\}.$ It follows that
$17342\leq (n-1)(n-2)/2$ and hence $n\geq 188.$ But then
$d_2(S_n)\geq 17390>17342.$ Thus $d_1(S_n)=n-1=17342$ hence
$n=17343$ and so $d_2(S_n)\geq 150363810>114954,$ a contradiction.

\noindent $(6)$ $S=M.$ Then $p(S)=71$ and $n\geq 2p(S)=142.$ By
Lemma \ref{lem1}, $d_4(S_n)\geq 457169.$ As
$\{196883,21296876\}\subseteq cd(S),$ we obtain
$\{393766,42593752\}\subseteq cd(S_n).$ Since $d_4(S_n)>393766,$ we
have $393766\in \{d_1(S_n),d_2(S_n),d_3(S_n)\}.$ It follows that
$393766\leq (n-1)(n-2)/2$ and hence $n\geq 889.$ As $d_2(S_n)\geq
393827>393766,$ we have $d_1(S_n)=n-1=393766$ hence $n=393767$ and
so $d_2(S_n)\geq 77525634494>42593752>393766,$ a contradiction. The
proof is complete.
\end{proof}

{\bf Proof of Theorem \ref{main}.} Suppose that $X_1(G)=X_1(S_n).$
It follows that $|G|=|S_n|=n!,$ $|G:G'|=2,$ $k(G)=k(S_n)$ and
$cd(G)=cd(S_n).$ For $n\leq 3,$ the result is trivial. If $n=4,$
then the result follows from  \cite[Chapter $17,$ Exercise
$2$]{Ber1}. Thus from now on, we assume that $n\geq 5.$

We first show that $G'=G''.$ By way of contradiction, assume that
$G''<G'.$ Let $N\leq G'$ be a normal subgroup of $G$ maximal such
that $G/N$ is solvable and $G'/N$ is the unique minimal normal
subgroup of $G/N.$ By \cite[Lemma $12.3$]{Isaacs}, all non-linear
irreducible characters of $G/N$ have equal degree $f$ and either
$G/N$ is a $p$-group, $Z(G/N)$ is cyclic and $G/N/Z(G/N)$ is
elementary abelian of order $f^2$ or $G/N$ is a Frobenius group with
an abelian Frobenius complement of order $f,$ and $G'/N$ is the
Frobenius kernel and is an elementary abelian $p$-group. Assume
first that $G/N$ is a $p$-group. As $G/N/Z(G/N)$ is abelian, we have
$G'/N\leq Z(G/N).$ Since $|G/N:G'/N|=2$ and $G/N$ is non-abelian, we
deduce that $G'/N=Z(G/N)$ and so $G/N/Z(G/N)$ is a cyclic group of
order $2,$ which is a contradiction as $G/N/Z(G/N)$ is elementary
abelian of order $f^2.$ Thus the second situation holds. It follows
that $f=|G/N:G'/N|=|G:G'|=2.$ Therefore $2\in cd(G)=cd(S_n),$ which
is impossible as the minimal non-trivial irreducible character
degree of $S_n$ is  $n-1\geq 4$ as $n\geq 5.$

Let $M\leq G'$ be a normal subgroup of $G$ so that $G'/M$ is a chief
factor of $G$ and so $G'/M\cong S^k,$ where $S$ is a non-abelian
simple group where $k\geq 1.$ Let $C/M=C_{G/M}(G'/M).$ Then $M\leq
C\unlhd G.$

Assume first that $C=M.$ Then $G'/M$ is the unique minimal normal
subgroup of $G/M.$ Since $|G/M:G'/M|=|G:G'|=2,$ we deduce that $k$
is at most $2.$ However $k$ cannot be $2$ by Theorem \ref{th2}. Thus
$k=1$ and so $G/M$ is an almost simple group with socle $G'/M$ and
$cd(G/M)\subseteq cd(S_n).$ By Theorem \ref{th1}, we have $G'/M\cong
A_n.$ It follows that $|G/M|=2|G'/M|=n!=|S_n|$ and so $M=1$ as
$|G|=|S_n|.$ Thus $G$ is an almost simple group with socle $A_n$ and
$|G|=n!.$ If $n\neq 6$ then as $Aut(A_n)=S_n$ we deduce that $G\cong
S_n.$ Now assume that $n=6.$ Then $G$ is isomorphic to one of the
following groups $S_6\cong A_6.2_1, PGL_2(9)\cong A_6.2_2\mbox{ and
} M_{10}\cong A_6.2_3.$ Using \cite{atlas}, we can see that $G$ must
be isomorphic to $S_6.$

 Finally assume that $C\neq M.$
It follows that $C/M$ is a non-trivial normal subgroup of $G/M$ and
so $C/M\cap G'/M$ is trivial. Thus $G'<G'C\leq G.$ Since $|G:G'|=2,$
we deduce that $G=G'C$ and hence $G/M=G'/M\times C/M,$ where $C/M$
is a cyclic subgroup of order $2.$ Thus $cd(G'/M)=cd(G/M)\subseteq
cd(S_n).$ Applying Theorem \ref{th1} again, we obtain $G'/M\cong
A_n,$ and hence $G/M\cong A_n\times \Z_2.$ By comparing the orders,
we deduce as in previous case that $M=1$ and so $G\cong A_n\times
\Z_2.$ We now show that this case cannot happen. In fact, we have
$k(G)=2k(A_n),$ and hence it suffices to show that $k(S_n)<2k(A_n).$
 Let $\lambda$ be a partition of
$n$ and denote by $\chi^\lambda$ the irreducible character of $S_n$
associated to $\lambda.$ If $\lambda$ is not self-conjugate, that is
$\lambda\neq \lambda',$ where $\lambda'$ denotes the conjugate of
$\lambda,$ then $({\chi^\lambda})_{A_n}=({\chi^{\lambda'}})_{A_n}\in
Irr(A_n).$ Otherwise, $({\chi^\lambda})_{A_n}=\chi^{\lambda
+}+\chi^{\lambda -},$ where $\chi^{\lambda +},\chi^{\lambda -}\in
Irr(A_n)$ are two non-equivalent irreducible characters of the same
degree. Let $p_s(n)$ be the number of self-conjugate partitions of
$n.$ Then $k(A_n)=(k(S_n)-p_s(n))/2+2p_s(n).$ Hence
$k(S_n)=2k(A_n)-3p_s(n).$ So it suffices to show that $p_s(n)\geq 1$
for $n\geq 5.$  In fact, if $n=2l+1,$ then we take
$\lambda=(l+1,1^l)$ and if $n=2l$ then take $\lambda=(l,2,1^{l-1}).$
Then $\lambda$ is a self-conjugate partition of $n$ so that
$p_s(n)\geq 1.$ This finishes the proof. $\hfill\square$

\subsection*{Acknowledgment} The author is grateful to Prof. Jamshid
Moori for his help with the preparation of this work.

\end{document}